\documentclass[a4paper,11pt]{amsart}
\usepackage{hyperref,fancyhdr,mathrsfs,amsmath,amscd,amsthm,amsfonts,latexsym,amssymb,stmaryrd}
\usepackage[all]{xy}

\DeclareMathOperator{\dif}{d}

\newcommand{\Cal}{\mathcal{C}}

\newcommand{\ol}{\mathcal{O}}

\def \g{\gamma}

\def \G{\Gamma}
\def \l{\lambda}

\def \phi{\varphi}
\def \Phi{\varPhi}
\def \p{\pi}
\def \r{\rho}
\def \s{\sigma}
\def \t{\tau}

\def \R{\mathbb{R}}
\def \Hq{\mathbb{H}\,}
\def \C{\mathbb{C}\,}

\def\widecheckg{g^{\hspace*{-2.5pt}\vbox to 5pt{\hbox to
0pt{\LARGE$\check{}$}}}\hspace*{2pt}}

\def\widecheckl{\lambda^{\hspace*{-3.5pt}\vbox to 8pt{\hbox to
0pt{\LARGE$\check{}$}}}\hspace*{2pt}}

\begin{document}

\title{The Penrose transform in quaternionic geometry} 
\author{Radu Pantilie}
\email{\href{mailto:radu.pantilie@imar.ro}{radu.pantilie@imar.ro}}
\address{R.~Pantilie, Institutul de Matematic\u a ``Simion~Stoilow'' al Academiei Rom\^ane,
C.P. 1-764, 014700, Bucure\c sti, Rom\^ania}
\subjclass[2010]{Primary 53C28, Secondary 53C26}
\keywords{the Penrose transform, twistor theory, quaternionic geometry} 

\newtheorem{thm}{Theorem}[section]
\newtheorem{lem}[thm]{Lemma}
\newtheorem{cor}[thm]{Corollary}
\newtheorem{prop}[thm]{Proposition}

\theoremstyle{definition}

\newtheorem{defn}[thm]{Definition}
\newtheorem{rem}[thm]{Remark}
\newtheorem{exm}[thm]{Example}

\numberwithin{equation}{section}

\begin{abstract}
We study the Penrose transform for the `quaternionic objects' whose twistor spaces are complex manifolds 
endowed with locally complete families of embedded Riemann spheres with positive normal bundles.
\end{abstract} 

\maketitle
\thispagestyle{empty}

\vspace{-7mm}  

\section*{Introduction}

\indent 
The obvious aim, when trying to solve an equation, is to find all the solutions, explicitly. 
However, one is usually satisfied with a subclass of solutions and a procedure by which 
all the solutions can be obtained from those in the subclass.\\ 
\indent 
The (arche)typical example is given by the Laplacian on $\R^2(=\!\C)$, where any complex-valued harmonic 
function is, locally, the sum of a holomorphic and an anti-holomorphic function.\\ 
\indent 
On $\R^3$, the situation was settled, essentially, in \cite{Whi-1903}\,: for any complex-valued 
harmonic function $f$\,, locally defined on $\R^3$, there exists a family of maps 
$\phi_t:\R^3\to\C$, with $t\in I$\,, and $I\subseteq\R$ a closed interval, such that $\phi_t=\psi_t\circ\p_t$\,, with 
$\p_t:\R^3\to\C$ orthogonal projections, $\psi_t$ locally defined holomorphic functions on $\C$, and 
\begin{equation} \label{e:Penrose_prehistory}   
f=\int_I\phi_t\dif\!t\;. 
\end{equation} 
\indent 
On $\R^4$, with the Lorentzian signature, this was done by R.~Penrose (see \cite{Pen-69} and the references therein). 
Up to a complexification (in Twistor Theory, the objects are usually real-analytic) we may assume Riemannian signature.\\ 
\indent  
The procedure by which all of the harmonic functions on $\R^4$ can be found is given, again, by \eqref{e:Penrose_prehistory}\,, 
whilst the particular class of solutions is formed of all the functions $\phi$ for which there exists a (positive) K\"ahler structure on $\R^4$,  
with respect to which $\phi$ is holomorphic.\\ 
\indent 
The first example, above, admits the obvious higher dimensional generalization through the classical notion of pluriharmonic function, of complex analysis. 
Furthermore, the unicity is easy to deal with.\\ 
\indent 
On the other hand, neither the unicity nor the adequate level of generality can be easily described for the other two examples. 
Nevertheless, the former problem can be understood through the classical fact that the integral representations above can be obtained, through the \v Cech cohomology, 
from elements of $H^1\bigl(Z,\mathcal{L}^2\bigr)$\,, where $Z=\ol(2)$\,, for $\R^3$, $Z=\ol(1)\oplus\ol(1)$\,, for $\R^4$, 
and $\mathcal{L}=\p^*\bigl(\ol(-1)\bigr)$\,, with $\ol(k)$ the holomorphic line bundle of Chern number $k$ over the Riemann sphere, 
and $\p$ the bundle projection of $Z$. That is, \emph{the Penrose transform} establishes a correspondence between certain 
cohomology classes over the twistor space $Z$ of $M$ and the `pluriharmonic functions' on~$M$.\\ 
\indent 
In this paper, we study the Penrose transform in the case when, up to conjugations, the twistor space is a complex manifold endowed 
with a locally complete family of embedded Riemann spheres with positive normal bundles. On the differential geometric side,  
these correspond to what we call `quaternionic objects' (see Section \ref{section:q_objects}\,, below). This way, we retrieve, for example, 
the already known cases of (classical) quaternionic manifolds \cite{Bas-q_complexes}\,, and, in particular, the anti-self-dual manifolds \cite{Hit-80}\,.\\ 
\indent 
The fact that the parameter spaces of families of embedded Riemann spheres are, canonically, `quaternionic' has its roots in 
the fact that the quaternions arise from the Riemann sphere endowed with the antipodal map. This is explained in Section \ref{section:quaternions} 
which we believe will help the reader to understand the quaternionic objects and the corresponding Penrose transform.\\ 
\indent 
The relevant notion of `quaternionic pluriharmonicity' is constructed in Section \ref{section:q-ph}\,. This is done in two steps, 
by using the fact that any suitable holomorphic line bundle over the twistor space of $M$ corresponds to a 
`hypercomplex object' which is the total space of a principal bundle over $M$ (see Definition \ref{defn:hypercomplex_object}\,, 
and the comments before it).\\ 
\indent 
Finally, the quaternionic Penrose transform is discussed in Section \ref{section:qPt}\,. We believe that our approach is 
as explicit as possible, based on a cohomology result of \cite{EGW}\,, and in the particular case of anti-self-dual manifolds 
it, essentially, reduces to the approach of \cite{W'd'se-85} (see, also, \cite{Raw-79}\,).

\section{Where do the quaternions come from?} \label{section:quaternions} 

\indent 
The idea is that \emph{the quaternions arise from the Riemann sphere} (endowed with the antipodal map), as we shall now explain.\\  
\indent 
A \emph{Riemann sphere} is a compact Riemann surface $Y$, with zero first Betti number, 
endowed with an involutive antiholomorphic diffeomorphism $\s$ without fixed points. 
From the exact sequence of cohomology groups associated to the exponential sequence 
$\{0\}\longrightarrow\mathbb{Z}\longrightarrow\C\longrightarrow\C^{\!*}\longrightarrow\{1\}$ 
we deduce that $H^1(Y,\C^{\!*})=\mathbb{Z}$\,. In particular, each holomorphic line bundle over $Y$ is determined,  
up to isomorphisms, by its Chern number. Furthermore, each element of $H^1(Y,\C^{\!*})$ is determined by a divisor.   
Thus, on denoting by $\overline{Y}$ the Riemann surface with the same underlying conformal structure as $Y$ 
but with the opposite orientation, $\s$ induces an isomorphism $H^1(Y,\C^{\!*})=H^1(\overline{Y},\C^{\!*})$ which 
preserves the Chern numbers.\\ 
\indent 
On the other hand, if $\mathcal{L}$ is a holomorphic line bundle over $Y$ then its conjugate $\overline{\mathcal{L}}$ 
is a holomorphic line bundle over $\overline{Y}$. Consequently, $\mathcal{L}$ is isomorphic to $\s^*(\overline{\mathcal{L}})$\,. 
Equivalently, there exists an antiholomorphic diffeomorphism $\t:\mathcal{L}\to\mathcal{L}$\,, covering $\s$, which is 
complex-conjugate linear on each fibre. Moreover, any two such morphisms differ by the multiplication with a nonzero complex number.  
Furthermore, as $\t^2$ is an isomorphism of $\mathcal{L}$ (covering the identity map of $Y$) we have $\t^2=\l\,{\rm Id}_{\mathcal{L}}$\,, 
for some nonzero complex number $\l$\,. But $\t^2$ is, also, an isomorphism of $\overline{\mathcal{L}}$ which implies that $\l$ is real. 
Therefore, by suitably multiplying, if necessary, $\t$ with a complex number, we may assume $\t^2=\pm{\rm Id}_{\mathcal{L}}$\,.    
Note that, $\t$ is unique, with this property, up to a factor of modulus~$1$. 

\begin{prop}[cf.~\cite{Qui-QJM98}\,] \label{prop:Q-J}
$\t^2=(-1)^d\,{\rm Id}_{\mathcal{L}}$\,, where $d$ is the Chern number of $\mathcal{L}$\,. 
\end{prop} 
\begin{proof} 
Suppose that $\mathcal{L}$ has Chern number $1$. By using the Kodaira vanishing and the Riemann--Roch theorems,  
we obtain that $H=H^0(Y,\mathcal{L})$ has (complex) dimension two. Consequently, the elements of $H$ which vanish at a point of $Y$ 
form a one-dimensional subspace of $H$. Therefore we may identify $Y=PH$ and, together with the fact that $\s$ has no fixed points, 
we deduce $\t^2=-{\rm Id}_{\mathcal{L}}$\,.\\ 
\indent 
Now, the proof follows quickly from the fact that any holomorphic line bundle over $Y$, of Chern number $d$\,,  
is of the form $\mathcal{L}^d$, where $\mathcal{L}$ has Chern number $1$. 
\end{proof} 

\begin{rem} \label{rem:Q-J}  
Let $\mathcal{L}$ be a holomorphic line bundle, of Chern number $-1$, over a Riemann sphere $Y$, and let $H=H^0(Y,\mathcal{L}^*)$\,. 
The proof of Proposition \ref{prop:Q-J} shows that, under the identification $Y=PH(=P(H^*)\,)$\,, 
$\mathcal{L}$ becomes the tautological line bundle, canonically embedded into the trivial bundle $Y\times H^*$. 
\end{rem} 

\indent 
We can now make the following: 

\begin{defn} 
Let $Y$ be the Riemann sphere and let $H=H^0(Y,\mathcal{L})$\,, where $\mathcal{L}$ is a holomorphic line bundle over $Y$ of Chern number $1$.\\ 
\indent 
The \emph{algebra of quaternions} is the unital associative subalgebra $\Hq\!\subseteq{\rm End}H$ formed of the elements 
which commute with the complex-conjugate isomorphism induced by~$\t$. (Briefly, $\Hq$ is formed of the `real' elements of ${\rm End}H$.) 
\end{defn} 

\indent 
The real multiples of ${\rm Id}_H$ are the \emph{real quaternions}, whilst the trace-free elements of $\Hq\!\subseteq{\rm End}H$ are 
the \emph{imaginary quaternions}. Also, the (complex) determinant on ${\rm End}H$ restricts to give the Euclidean structure of $\Hq$. 
Then the imaginary quaternions of modulus $1$ are linear complex isomorphisms of $H$ of square $-{\rm Id}_H$\,. Thus, we obtain 
the identification $S^2=PH(=Y)$ through which any imaginary quaternion of modulus $1$ is identified with its eigenspace 
corresponding to $-{\rm i}$\,. Then, under this identification, $\s$ is the antipodal map.\\  
\indent 
Note that, $\Hq$ does not depend of the choice of $\t$, but only of $Y$, $\s$, and $\mathcal{L}$\,. 
Moreover, if $\mathcal{L}'$ is another holomorphic line bundle, of Chern number $1$, and $H'=H^0(Y,\mathcal{L}')$  
then the two embeddings of $\Hq$ into ${\rm End}H$ and ${\rm End}H'$ differ by a composition with an element of  
${\rm SO}(3)$ acting trivially on $\R$ and canonically on $\R^3(={\rm Im}\Hq)$. 
In particular, the automorphism group of $\Hq$ is ${\rm SO}(3)$\,.

\section{The main objects of quaternionic geometry}  \label{section:q_objects}  

\indent 
In this section, we define the class of manifolds for which we want to describe the Penrose transform. 
Up to integrability, the corresponding geometric structure is given by the almost co-CR quaternionic structures of \cite{fq,fq_2}\,, 
whilst to formulate the integrability we need the generalized connections of \cite{qgfs}\,. We start by recalling the latter. 

\subsection{Principal $\r$-connections}  
Let $(P,M,G)$ be a principal bundle and let $F$ be a vector bundle over $M$, endowed with a morphism of vector bundles $\r:F\to TM$. 
A \emph{principal $\r$-connection} on $P$ is a $G$-invariant morphism of vector bundles $C:\p^*F\to TP$ such that $\widetilde{\dif\!\p}\circ C=\p^*\r$\,, 
where $\p:P\to M$ is the projection and $\widetilde{\dif\!\p}:TP\to\p^*(TM)$ is the morphism of vector bundles induced by $\dif\!\p$\,.\\ 
\indent 
The notion of associated connection has a straight generalization to this setting, as we shall now explain. Let $S$ be a manifold on which 
$G$ is acting to the left. Denote by $Y=P\times_GS$ the associated bundle, by $\p_Y$ its projection, and by $\mu:TP\times S\to TY$ 
the obvious $G$-invariant morphism of vector bundles, covering the projection $P\times S\to Y$.\\ 
\indent 
On identifying, as usual, $\p_Y^*(F)$ with a submanifold of $Y\times F$, we may define the \emph{associated $\r$-connection} $c:\p_Y^*(F)\to TY$, 
as follows: $c\bigl([u,s],f\bigr)=\mu\bigl(C(u,f),s\bigr)$\,, for any $(u,f)\in\p^*F$ and $s\in S$, 
where $[u,s]\in Y$ is the equivalence class determined by $(u,s)\in P\times S$.  Note that, if we compose $c$ with the morphism of vector bundles  
from $TY$ onto $\p_Y^*(TM)$\,, induced by $\dif\!\p_Y$, we obtain $\p_Y^*(\r)$\,. Indeed, as $\dif\!\p_Y\circ\mu$ is the composition 
of the projection from $TP\times S$ onto $TP$ followed by $\dif\!\p$\,, we have $(\dif\!\p_Y\circ c)\bigl([u,s],f\bigr)=(\dif\!\p\circ C)(u,f)=\r(f)$\,, 
for any $(u,f)\in\p^*F$ and $s\in S$. 

\begin{rem} 
Let $(P,M,G)$ be a principal bundle endowed with a principal $\r$-connection $C:\p^*F\to TP$, where $F$ is a vector bundle over $M$, 
and $\r:F\to TM$ a morphism of vector bundles.\\ 
\indent 
Suppose that $E\subseteq F$ is a vector subbundle mapped isomorphically by $\r$ onto a distribution on $M$. Then, by restriction, $C$ induces 
a partial principal connection on $P$ over $\r(E)$\,. This, obviously, works, also, if $E\subseteq F^{\C\!}$ is a complex vector subbundle 
mapped isomorphically by (the complexification of) $\r$ onto a complex distribution on $M$. 
\end{rem}  

\indent 
If $S$ is a vector space on which $G$ acts by linear isomorphisms then any (associated) $\r$-connection on the associated vector bundle 
$Y=P\times_GS$ corresponds to a covariant derivation $\nabla:\G(Y)\to\G\bigl({\rm Hom}(F,Y)\bigr)$ which is a linear map satisfying 
$\nabla(fs)=\r^*(\dif\!f)\otimes s+f(\nabla s)$\,, for any function $f$ on $M$ and any section $s$ of $Y$. 
If $Y=F$ then we can define the \emph{torsion} of $\nabla$ which is the section $T$ of $TM\otimes\Lambda^2F^*$ characterised by 
$T(s_1,s_2)=\r\bigl(\nabla_{s_1}s_2-\nabla_{s_2}s_1\bigr)-\bigl[\r(s_1),\r(s_2)\bigr]$\,, for any sections $s_1$\,, $s_2$ of $F$.  

\subsection{Quaternionic objects} 
A \emph{linear quaternionic structure} on a (real) vector space $F$ is an equivalence class of morphisms of 
associative algebras from $\Hq$ to ${\rm End}F$, where two such morphisms $\s_1$ and $\s_2$ are equivalent if $\s_2=\s_1\circ a$\,, for some 
$a\in{\rm SO}(3)$\,. If $\s:\Hq\to{\rm End}F$ is a morphism of associative algebras then the induced linear quaternionic structure on $F$ 
is determined by $\s(S^2)$\,, whose elements are the \emph{admissible linear complex structures} on $F$. Thus, a linear quaternionic structure is
given by a family of linear complex structures, parameterized by the Riemann sphere.\\ 
\indent 
A \emph{linear co-CR quaternionic structure} on a vector space $U$ is a pair $(F,\r)$\,, where $F$ is a quaternionic vector space, and $\r:F\to U$ 
is a surjective linear map such that $({\rm ker}\r)\cap J({\rm ker}\r)=\{0\}$\,, for any admissible linear complex structure $J$ on $F$; consequently, 
$C=\r\bigl({\rm ker}(J+{\rm i})\bigr)$ is a linear co-CR structure on $U$ (that is, $C+\overline{C}=U^{\C}$). Thus, a linear co-CR quaternionic structure  
is given by a family of linear co-CR structures, parameterized by the Riemann sphere.\\ 
\indent 
These notions extend in the obvious way to vector bundles, thus, giving the notion of almost co-CR quaternionic structure on a manifold. 
Note that, if $F$ is a quaternionic vector bundle then the corresponding space $Y$ of admissible linear complex structures on (the fibres of) $F$ 
is a sphere bundle, with structural group ${\rm SO}(3)$\,.\\ 
\indent  
Now, let $(F,\r)$ be an almost co-CR quaternionic structure on $M$ and suppose that $F$ is endowed with a $\r$-connection (compatible with its  
structural group). Denote by $Y$ the bundle of admissible linear complex structures on $F$ and let $c:\p^*F\to TY$ be the associated $\r$-connection on it, 
where $\p:Y\to M$ is the projection.\\ 
\indent 
Let $\mathcal{B}\subseteq T^{\C\!}Y$ be the complex distribution given by $\mathcal{B}_J=c_J\bigl({\rm ker}(J+{\rm i})\bigr)$\,, for any $J\in Y$. 
Then $\Cal=\mathcal{B}\oplus({\rm ker}\dif\!\p)^{0,1}$ is an almost co-CR structure on $Y$; that is, $\Cal+\overline{\Cal}=T^{\C\!}Y$. 

\begin{defn} \label{defn:q_object} 
We say that $(M,F,\r,c)$ is a \emph{quaternionic object} if $\Cal$ is integrable (that is, its space of sections is closed under the usual bracket). 
\end{defn} 

\indent 
Note that, a quaternionic object is a (real) $\r$-quaternionic manifold \cite{qgfs} with $\r$ surjective. In particular, if $\r$ is an isomorphism then we obtain 
the classical notion of quaternionic manifold.\\ 
\indent 
To define the corresponding twistor spaces, suppose that $(M,F,\r,c)$ is a quaternionic object for which there exists a surjective submersion 
$\psi:Y\to Z$ such that:\\ 
\indent 
\quad(1) $({\rm ker}\dif\!\psi)^{\C}=\Cal\cap\overline{\Cal}$\,,\\ 
\indent  
\quad(2) $\Cal$ is projectable with respect to $\psi$ (note that, this is a consequence of (1)\,, if the fibres of $\psi$ are connected),\\ 
\indent 
\quad(3) $\psi$ restricted to each fibre of $\p$ is injective.\\ 
\indent 
Then $Z$ endowed with $\dif\!\psi(\Cal)$ is a complex manifold (with $T^{0,1}Z=\dif\!\psi(\Cal)$\,) which is the \emph{twistor space} 
of $(M,F,\r,c)$\,. Furthermore, $\bigl(\psi(Y_x)\bigr)_{x\in M}$\,, where $Y_x=\p^{-1}(x)$\,, is a smooth family of embedded Riemann spheres on $Z$\,,   
whose members are the \emph{real twistor spheres}. 

\begin{exm} 
Let $(U,F,\r)$ be a co-CR quaternionic vector space. Then $U\times F$ is a quaternionic vector bundle over $U$ and ${\rm Id}_U\times\r$ 
is a morphism of vector bundles from $U\times F$ onto $TU\,(=U\times U)$\,, giving an almost co-CR quaternionic structure on $U$.\\ 
\indent 
Moreover, on endowing $U\times F$ with the obvious (classical) flat connection, we obtain a quaternionic object whose twistor space is 
the holomorphic vector bundle $\mathcal{U}$\,, over the Riemann sphere, which is the quotient of the trivial holomorphic vector bundle $S^2\times U^{\C}$ 
through $\bigl(\{J\}\times\r\bigl({\rm ker}(J+{\rm i})\bigr)\bigr)_{J\in S^2}$\,; furthermore, $\mathcal{U}$ determines $(U,F,\r)$ \cite{fq,fq_2}\,. 
\end{exm} 

\indent 
Returning to the general case, let $Z$ be the twistor space of the quaternionic object $(M,F,\r,c)$\,. Then, for any $x\in M$, the normal 
bundle of the corresponding twistor sphere $\psi(Y_x)\subseteq Z$\,, where $\psi:Y\to Z$ is as above, is (isomorphic to) 
the holomorphic vector bundle of $(T_xM,E_x,\r_x)$\,. Consequently, by applying classical results (see \cite{qgfs} and the references therein) 
we obtain that $\bigl(\psi(Y_x)\bigr)_{x\in M}$ is contained by a locally complete family of embedded Riemann spheres parameterized 
by a complexification of $M$ (in particular, $M$ is real-analytic); the members of this family are the \emph{twistor spheres}. 
Conversely, up to a conjugation and on restricting, if necessary, to an open subset, \emph{any complex manifold endowed with a locally complete family 
of Riemann spheres, with positive normal bundles, is the twistor space of a quaternionic object} (consequence of \cite{qgfs}\,).\\ 
\indent 
There is only one more ingredient needed to start describing the Penrose transform for quaternionic objects: a holomorphic line bundle 
over the twistor space.\\ 
\indent 
Let $Z$ be the twistor space of the quaternionic object $(M,F,\r,c)$\,. Denoting by $K_Z$ the canonical line bundle of $Z$ (that is, $K_Z$ is the determinant 
of the holomorphic cotangent bundle of $Z$) and with $\psi:Y\to Z$ as above, we may assume, by passing, if necessary, to an open neighbourhood 
of any point of $M$, that there exists a line bundle $\mathcal{L}$ over $Z$ which is a root of $K_Z$ and whose restriction 
to each twistor sphere has Chern number $-1$. Furthermore, as the antipodal map on $Y$ induces an involutive antiholomorphic involution $\s:Z\to Z$, 
without fixed points, we have that $\s$ is covered by an antiholomorpic involution $\t_Z:K_Z\to K_Z$\,. 
It follows that, we may assume that there exists an antiholomorphic anti-involution $\t:\mathcal{L}\to\mathcal{L}$ covering $\s$\,. 
Note that, if we restrict $\mathcal{L}$ to the real twistor spheres, then $\t$ give antiholomorphic anti-involutions as in Section \ref{section:quaternions}\,. 
Also, for classical quaternionic manifolds of dimensions at least eight it is known that $\mathcal{L}^2$ exists globally \cite{Sal-dg_qm}\,, \cite{PePoSw-98}\,. 
 
\begin{rem}  
\emph{The quaternionic objects are abundant.} Various classes of examples can be found in \cite{Pan-integrab_co-cr_q}\,, \cite{qgfs}\,. 
Furthermore, if the twistor space of a quaternionic object 
is a complex projective manifold then it is a `rationally connected manifold', one of the main objects of study of algebraic geometry 
(see \cite{Paltin-2005}\,, and the references therein). 
\end{rem}

\section{Quaternionic pluriharmonicity} \label{section:q-ph} 

\indent 
Let $(M,F,\r,c)$ be a quaternionic object with twistor space $Z$. Suppose that $\mathcal{L}$ is a holomorphic line bundle over $Z$ 
whose restriction to each twistor sphere has Chern number $-1$, and which is endowed with an anti-holomorphic anti-involution $\t$ 
which covers the conjugation $\s$ on $Z$ induced by the antipodal map on $Y$. As already mentioned, such a line bundle 
always exists locally (that is, if we pass to an open subset of $M$). 
Let $H$ be the dual of the direct image through $\p$ of $\psi^*(\mathcal{L}^*)$\,; 
that is, the space of sections of $H^*$ over each open set $U\subseteq M$ is the space of 
sections of $\psi^*(\mathcal{L}^*)$\,, over $\p^{-1}(U)$\,, which are holomorphic when restricted to the fibres of $\p$.\\ 
\indent 
We may assume $F^{\C\!}=H\otimes E$, where $E$ is a complex vector bundle. 
Indeed, the tensor product of $\psi^*(\mathcal{L}^*)$ and $\mathcal{B}$ 
(note that, the latter is equal to the quotient of $({\rm ker}\dif\!\psi)^{\C}$ 
through $({\rm ker}\dif\!\p)^{0,1}$) is trivial when restricted to each fibre of $\p$\,. Therefore 
$\psi^*(\mathcal{L}^*)\otimes\mathcal{B}=\p^*(E^*)$\,, for some vector bundle $E$ over $M$. 
Now (cf.~\cite{qgfs}\,), $F^{\C\!}$ is equal to the dual of the direct image through $\p$ of $\mathcal{B}^*=\psi^*(\mathcal{L}^*)\otimes\p^*(E)$ 
and is therefore equal to $H\otimes E$.\\  
\indent 
Note that (compare Remark \ref{rem:Q-J}), $Y=PH$, and $\psi^*(\mathcal{L}\setminus0)=H\setminus0$\,, as principal bundles over $Y$, 
with structural group $\C^{\!*}$. We shall denote by $\psi_H:H\setminus0\to\mathcal{L}\setminus0$ the corresponding bundle-map.\\  
\indent 
Furthermore, $H\setminus0$ is a principal bundle over $M$, with group $\Hq^{\!*}$; that is, $H$ is a hypercomplex vector bundle of real rank four over $M$. 
Indeed, for any $u\in H\setminus0$ and $z_1+{\rm j}z_2\in\Hq^{\!*}$, we may define $u\cdot(z_1+{\rm j}z_2)=z_1u+z_2\t u$\,.\\ 
\indent 
Let $\p_H:H\to M$ and $p:H\setminus0\to Y$ be the projections. We have the following diagram:    
\begin{equation}  \label{diagram:for_Penrose} 
\begin{gathered} 
\xymatrix{
                                                       &              H\setminus0 \ar[dl]_{\psi_H} \ar[d]^p  \ar@/^/[ddr]^{\p_H}       &                  \\
    \mathcal{L}\setminus0  \ar[d]   &                     \hspace{4mm}Y(=PH)    \ar[dl]_{\psi} \ar[dr]^{\p}      &                  \\
                      Z                               &                                                                 &            M                             
   } 
\end{gathered} 
\end{equation}

\indent 
Obviously, $\Cal_H=(\dif\!\psi_H)^{-1}\bigl(T^{0,1}(\mathcal{L}\setminus0)\bigr)$ is a co-CR structure on $H\setminus0$\,. 
(Note that, $\dif\!p(\Cal_H)$ is the co-CR structure involved in Definition \ref{defn:q_object}\,.)  
If $q\in\Hq^{\!*}$ we denote by $[q]$ its image under the projection onto $\Hq^{\!*}\!/\C^{\!*}=\C\!P^1$ (where $\C^{\!*}\subseteq\Hq^{\!*}$ 
is acting to the right on $\Hq^{\!*}$). Then $[q]\mapsto\Cal_H\cdot q^{-1}$ is a family of co-CR structures on $H\setminus0$ parametrized 
by the Riemann sphere. It follows that $H\setminus0$ is a quaternionic object such that its twistor space $Z(H\setminus0)$\,, 
with respect to its underlying smooth structure, is diffeomorphic to $S^2\times(\mathcal{L}\setminus0)$\,. 
Furthermore, $\Hq^{\!*}$ acts by twistorial diffeomorphisms on $H\setminus0$ (corresponding to holomorphic diffeomorphisms on $Z(H\setminus0)$\,).\\ 
\indent 
In fact, $H\setminus0$ is a special type of quaternionic object which we define, next. 

\begin{defn} \label{defn:hypercomplex_object} 
A quaternionic object $(M,F,\r,c)$ is \emph{hypercomplex} if the bundle $Y$ of admissible linear complex structures on $F$ 
is trivial and $c$ is the corresponding trivial flat connection. 
\end{defn} 

\indent 
For the twistor space $Z$ of a hypercomplex object $(M,F,\r,c)$\,, defined by $\psi:Y\to Z$, we shall assume that the projection $Y=M\times S^2\to S^2$ 
factorises into $\psi$ followed by a holomorphic submersion from $Z$ onto $S^2$ (obviously, this is automatically satisfied if the fibres of $\psi$ are 
connected).\\ 
\indent 
Note that, a hypercomplex object $(M,F,\r,c)$ with $\r$ an isomorphism is a classical hypercomplex manifold. 

\begin{prop} \label{prop:twistor_hyper}
Let $Z$ be the twistor space of a quaternionic object $(M,F,\r,c)$ defined by $\psi:Y\to Z$\,. If the fibres of $\psi$ are connected  
then the following assertions are equivalent:\\ 
\indent 
{\rm (i)} There exists a $\r$-connection $c_1$ with respect to which $(M,F,\r,c_1)$ is hypercomplex.\\ 
\indent 
{\rm (ii)} There exists a holomorphic submersion $\phi:Z\to S^2$ whose restriction to each twistor sphere is a diffeomorphism 
and which intertwines the conjugation on $Z$ and the antipodal map on $S^2$. 
\end{prop} 
\begin{proof} 
If (ii) holds, then $\phi\circ\psi:Y\to S^2$ induces a trivialization $M=Y\times S^2$. 
Conversely, if (i) holds then the fibres of $\psi$ are contained in the fibres of the projection $Y\to S^2$ induced by the trivialization $Y=M\times S^2$. 
Thus, $Y\to S^2$ factorises as $\phi\circ\psi$\,, where $\phi$ is as in (ii)\,.  
\end{proof} 

\indent 
For a hypercomplex object $(M,F,\r,c)$ with twistor space $Z$ we have a distinguished class of holomorphic line bundles over $Z$, namely 
those which are pull backs through $\phi$ of holomorphic line bundles over $S^2$. Any such holomorphic line bundle $\mathcal{L}$ 
is characterised by the fact that $\psi^{*\!}\mathcal{L}$ is trivial along the fibres of the projection $\phi\circ\psi:Y\to S^2$. 

\begin{prop} \label{prop:hyper_right_connection} 
Let $Z$ be the twistor space of a hypercomplex object $(M,F,\r,c)$\,, with a decomposition $F^{\C\!}=H\otimes E$ 
corresponding to a holomorphic line bundle $\mathcal{L}$ over $Z$.\\ 
\indent 
Then there exists a unique $\r$-connection on $H$ which induces the trivial flat connection on~$Y$ and which induces an `anti-self-dual' connection 
on $\Lambda^2H$ (that is, flat along the admissible co-CR structures defined by the constant sections of $Y(=M\times S^2)$\,). 
\end{prop} 
\begin{proof} 
On identifying $S^2=\C\!P^1$, assume, firstly, that $\mathcal{L}=\phi^*\bigl(\ol(-1)\bigr)$\,, and recall that the space of sections of $H^*$ 
over any open subset $U\subseteq M$ is, by definition, the space of sections of $\mathcal{L}^*$ over $\p^{-1}(U)$ 
which are holomorphic when restricted to the fibres of $\p$. We shall denote in the same way the local sections of $H^*$ 
and the corresponding local sections of $\mathcal{L}^*$.\\ 
\indent 
Let $X\in T_xM$,  $(x\in M)$, and let $\widetilde{X}$ be tangent to the fibres of $\phi\circ\psi$\,, along $\p^{-1}(x)$\,, 
and such that $\dif\!\p(\widetilde{X})=X$. As $\psi^{*\!}\mathcal{L}$ is trivial along the fibres of $\phi\circ\psi$ we 
may define $\widetilde{X}(s)$\,, for any local section $s$ of $H^*$, thus obtaining a holomorphic section of $\mathcal{L}^*$ along $\p^{-1}(x)$\,; 
that is, an element of $H^*_x$\,. We have, thus, obtained a connection on $H^*$ with respect to which the sections obtained as pull backs 
through $\phi\circ\psi$ of holomorphic sections of $\ol(1)$ are covariantly constant. As these sections generate each fibre of $H^*$, 
this shows that the obtained connection on $H^*$ is trivial. Dualizing, we obtain a trivializing flat connection on $H$.\\ 
\indent 
Now, let $\mathcal{L}_1$ be any other holomorphic line bundle over $Z$, whose restriction to each twistor sphere has Chern number $-1$, 
and endowed with an antiholomorphic anti-involution covering $\s$. 
Let $F^{\C\!}=H_1\otimes E_1$ be the induced decomposition. 
Then the Ward tranform admits a straight generalization to this setting to show that $\mathcal{L}_1\otimes\mathcal{L}^*$ 
corresponds to a complex line bundle $L$\,, over $M$, endowed with an anti-self-dual $\r$-connection. 
Consequently, $H_1=H\otimes L$\,, which on endowing with the tensor product $\r$-connection completes the proof.   
\end{proof}  

\begin{rem} \label{rem:line_bundle_of_H} 
Let $Z$ be the twistor space of a quaternionic object $(M,F,\r,c)$\,, with a decomposition $F^{\C\!}=H\otimes E$ 
corresponding to a holomorphic line bundle $\mathcal{L}$ over $Z$.\\ 
\indent  
As we have seen, $H\setminus0$ is a hypercomplex object so that there exists a holomorphic submersion $\phi$ from its twistor space 
$Z(H\setminus0)$ onto $\C\!P^1$.\\ 
\indent 
Also, $\p_H:H\setminus0\to M$ is a twistorial map, corresponding to a surjective holomorphic submersion from $Z(H\setminus0)$ 
onto $Z$. Then the pull back through this submersion of $\mathcal{L}$ is $\phi^*(\ol(-1))$\,. 
\end{rem}  

\indent 
By using Proposition \ref{prop:hyper_right_connection} and Remark \ref{rem:line_bundle_of_H}\,, we can now define the relevant 
`quaternionic pluriharmonic' sheaf, in two steps, as follows. 

\begin{defn} \label{defn:q_h_hyper} 
Let $(M,F,\r,c)$ be a hypercomplex object with twistor space $Z$, and a decomposition $F^{\C\!}=H\otimes E$ 
corresponding to a holomorphic line bundle $\mathcal{L}$ over~$Z$.\\ 
\indent  
Let $I$ and $J$ be anti-commuting admissible linear complex structures on $F$ given by constant sections of $Y(=PH=M\times S^2)$\,. 
Then $\Cal=\r\bigl({\rm ker}(I+{\rm i})\bigr)$ is a complex distribution on $M$ over which the $\r$-connection on $H$ 
induces a partial connection. As $\Cal$ is integrable, we may define the exterior covariant differential $\overline{\partial}$ mapping 
$(\Lambda^2H)$-valued $r$-forms on $\Cal$ to $(\Lambda^2H)$-valued $(r+1)$-forms on $\Cal$\,, $(r\in\mathbb{N})$\,.\\ 
\indent 
A section $f$ of $\Lambda^2H$ is \emph{hypercomplex pluriharmonic (with respect to the pair $(I,J)$)} if 
$\bigl(\overline{\partial}\circ J\circ\partial\bigr)(f)=0$\,. 
\end{defn} 

\indent 
For the particular case of (classical) hypercomplex manifolds, the pluriharmonic functions, also, appear in \cite{AleVer-hyper-psh}\,. 
However, the operator defining them is due to \cite{W'd'se-85}\,. 

\begin{rem} \label{rem:q_h_hyper} 
With the same notations as in Definition \ref{defn:q_h_hyper}\,, suppose that the connection on $H$ together with some $\r$-connection on $E$ 
induces a torsion free $\r$-connection $\nabla$ on $F$. 
Then a quick calculation shows that, for any section $f$ of $\Lambda^2H$, the following assertions are equivalent 
 (recall that, in dimension four, a linear quaternionic structure is just an oriented linear conformal structure):\\ 
\indent 
\quad(i) $f$ is hypercomplex pluriharmonic;\\  
\indent 
\quad(ii) $\nabla^{2\!}f$ restricted to any quaternionic subspace,  of real dimension four, of any fibre of $F$, 
is trace free.\\ 
\indent  
In particular, in this case, the notion of hypercomplex pluriharmonicity does not depend of the pair $(I,J)$\,.  
\end{rem}  

\begin{defn} \label{defn:q_h} 
Let $(M,F,\r,c)$ be a quaternionic object with twistor space $Z$, and a decomposition $F^{\C\!}=H\otimes E$ 
corresponding to a holomorphic line bundle $\mathcal{L}$ over~$Z$.\\ 
\indent 
A section $f$ of $\Lambda^2H$ is \emph{quaternionic pluriharmonic} if the corresponding equivariant function $\widetilde{f}:H\setminus0\to\C$  
is hypercomplex pluriharmonic. 
\end{defn} 

\indent
In the particular cases of hyper-K\"ahler and quaternionic manifolds, our definition reduces to the ones considered 
in \cite{HarLaw-some_psh} and \cite{Ale-q_pluripotential}\,, respectively (note that, the latter is based on \cite{Bas-q_complexes}\,).\\ 
\indent 
The fact that $\Hq^{\!*}$ acts by twistorial diffeomorphisms on $H\setminus0$ implies that Definition \ref{defn:q_h} 
does not depend of the choice of $I$ and $J$ on $H\setminus0$\,.

\section{The Penrose transform} \label{section:qPt} 

\indent 
Let $(M,F,\r,c)$ be a quaternionic object with twistor space $Z$, and a decomposition $F^{\C\!}=H\otimes E$ 
corresponding to a holomorphic line bundle $\mathcal{L}$ over~$Z$. We make the following definition, 
where $\psi:Y(=PH)\to Z$ is the corresponding surjective co-CR submersion (recall  the diagram \eqref{diagram:for_Penrose}\,).  

\begin{defn} 
The \emph{Penrose transform} is the complex linear map associating to each $\g\in H^1\bigl(Z,\mathcal{L}^2\bigr)$ the section  
$\bigl(\int_{\p^{-1}(x)}\psi^*\g\bigr)_{x\in M}$ of $\Lambda^2H$. 
\end{defn} 

\indent 
The injectivity of the Penrose transform is quite straightforward. 

\begin{prop} \label{prop:qPt_inj} 
If the fibres of $\psi$ are connected then, by passing to an open neighbourhood of each point of $M$, the Penrose transform is injective.  
\end{prop} 
\begin{proof} 
If $M$ is a hypercomplex object then this follows quickly from standard arguments, involving \v Cech cohomology and by using \cite{Siu-Stein_nbds}\,.  
In general, through the pull back, the proof for $H\setminus0$ gives, in particular, the required injectivity, locally, on $M$. 
\end{proof} 
  
\indent 
We say that the quaternionic object $(M,F,\r,c)$ is of \emph{constant type} if the co-CR quaternionic vector spaces $(T_xM,F_x,\r_x)$\,, $(x\in M)$\,, are the 
same, up to isomorphisms.\\ 
\indent 
Here is the main result of this paper.   

\begin{thm} \label{thm:q_Penrose} 
Let $(M,F,\r,c)$ be a quaternionic object with twistor space $Z$, and a decomposition $F^{\C\!}=H\otimes E$ 
corresponding to a holomorphic line bundle $\mathcal{L}$ over~$Z$.\\ 
\indent 
The following statements hold:\\ 
\indent 
\quad{\rm (i)} If on $H\setminus0$ the sheaf of quaternionic pluriharmonic functions is equal to the sheaf of hypercomplex pluriharmonic functions 
(in particular, if $H\setminus0$ admits a compatible torsion free $\r$-connection) 
then the image of the Penrose transform is contained in the space of quaternionic pluriharmonic sections of $\Lambda^2H$.\\ 
\indent 
\quad{\rm (ii)} If $(M,F,\r,c)$ is of constant type, then, by passing, if necessary, to an open neighbourhood of each point of $M$, 
the Penrose transform admits a partial section defined on the space of quaternionic pluriharmonic sections of $\Lambda^2H$. 
\end{thm}  

\indent 
The proof of Theorem \ref{thm:q_Penrose}\,, essentially, relies on the fact that it is sufficient to prove it for hypercomplex objects. 
To show this, we need the following result. 

\begin{prop} \label{prop:q_a_connection} 
Let $(M,F,\r,c)$ be a quaternionic object with twistor space $Z$, and a decomposition $F^{\C\!}=H\otimes E$ 
corresponding to a holomorphic line bundle $\mathcal{L}$ over~$Z$.\\ 
\indent 
Then there exists a principal $\r$-connection $C:\p_H^*F\to T(H\setminus0)$ on $(H\setminus0,M,\Hq^{\!*})$ such that the associated 
$\r$-connection on $Y(=PH)$ induces the same co-CR structure on $Y$ as the one defining $\psi:Y\to Z$. 
\end{prop} 
\begin{proof} 
This is a consequence of the fact that $\p_H:H\setminus0\to M$ is twistorial and in particular, at each point, its differential is a co-CR quaternionic linear map 
(see \cite{fq,fq_2} for the definition of the latter). Therefore, on denoting by $(F_H,\r_H)$ the almost co-CR quaternionic structure of $H\setminus0$ 
we have $\dif\!\p_H\circ\r_H=\r\circ\Pi$ where $\Pi:F_H\to F$ is a morphism of quaternionic vector bundles, covering $\p_H$\,. 
Furthermore, $\Pi$ factorises into a morphism of quaternionic vector bundles $\widetilde{\Pi}:F_H\to(\p_H)^*(F)$ followed by the bundle-map $(\p_H)^*(F)\to F$.\\ 
\indent 
Now, $(\p_H)^*\bigl(F^{\C\!}\bigr)=\bigl((H\setminus0)\times\C^{\!2}\bigr)\otimes E$, whilst $F_H^{\C\!}=\bigl((H\setminus0)\times\C^{\!2}\bigr)\otimes E_H$\,, 
for some complex vector bundle $E_H$ over $H\setminus0$\,, such that, $\widetilde{\Pi}$ is given by the tensor product of the identity morphism of 
$(H\setminus0)\times\C^{\!2}$ and a morphism of complex vector bundles from $E_H$ onto $(\p_H)^*(E)$\,. Hence, by taking a section of the latter, 
we may embedd $(\p_H)^*(F)\subseteq F_H$\,, as a quaternionic subbundle, so that the restriction of $\widetilde{\Pi}$ to it is the identity morphism.\\ 
\indent 
On the other hand, as a structural group, $\Hq^{\!*}$ induces right actions on $F_H$ and on $T(H\setminus0)$\,, 
where the former is given by the action on (the first components) of $(\p_H)^*(H)\bigl(=(H\setminus0)\times\C^{\!2}\bigr)$. 
Moreover, $\r_H$ is equivariant with respect to these actions.\\ 
\indent 
Then the restriction $C$ of $\r_H$ to $(\p_H)^*(F)$ is a principal $\r$-connection, as claimed.     
\end{proof} 

\indent 
Another ingredient needed to prove Theorem \ref{thm:q_Penrose} is a natural extension, to this setting, of the classical notion 
of `standard (horizontal) vector field'. For this, let $(M,F,\r,c)$ be a quaternionic object with twistor space $Z$, and a decomposition $F^{\C\!}=H\otimes E$ 
corresponding to a holomorphic line bundle $\mathcal{L}$ over $Z$. Endow $H\setminus0$ with a principal $\r$-connection $C$ as in 
Proposition \ref{prop:q_a_connection}\,.  

\begin{defn} \label{defn:q_standard} 
For any $\xi\in\C^{\!2}$ let $B(\xi)$ be the morphism of complex vector bundles, from $(\p_H)^*(E)$ to $T^{\C\!}(H\setminus0)$ 
given by $B(\xi)(u,s)=C\bigl(u,(u\xi)\otimes s\bigr)$\,, for any $u\in H\setminus0$ and $s\in E_{\p_H(u)}$\,. 
\end{defn} 
 
\begin{prop} \label{prop:q_standard_invariance} 
On denoting by $R_q$ the right translation by $q\in\Hq^{\!*}$ on $H\setminus0$ (and $(\p_H)^*(E)$\,), we have 
$\dif\!R_q\circ B(\xi)=B(q^{-1}\xi)\circ R_q$\,,  for any $\xi\in\C^{\!2}$. 
\end{prop} 
\begin{proof} 
Let $u\in H\setminus0$ and $s\in E_{\p_H(u)}$\,. By the equivariance of $C$, we have  
$$\dif\!R_q\bigl(B(\xi)(u,s)\bigr)=\dif\!R_q\bigl(C\bigl(u,(u\xi)\otimes s\bigr)\bigr)=C\bigl(u\cdot q,(u\xi)\otimes s\bigr)\;.$$  
\indent 
On the other hand, 
$$\bigl(B(q^{-1}\xi)\circ R_q\bigr)(u,s)=B(q^{-1}\xi)(u\cdot q,s)=C\bigl(u\cdot q, \bigl((u\cdot q)(q^{-1}\xi)\bigr)\otimes s\bigr)=C\bigl(u\cdot q,(u\xi)\otimes s\bigr)\;,$$ 
thus completing the proof. 
\end{proof} 

\indent 
The morphism of Lie algebras from $\Hq$ (seen as the Lie algebra of $\Hq^{\!*}$) to the Lie algebra of vector fields on $T(H\setminus0)$ extends 
to a morphism of complex Lie algebras from $A\in\mathfrak{gl}(2,\C\!)$ to the sections of $T^{\C\!}(H\setminus0)$\,. We shall denote in the same way, 
the elements of $\mathfrak{gl}(2,\C\!)$ and the corresponding (fundamental) complex vector fields on $H\setminus0$\,. 
Then, from Proposition \ref{prop:q_standard_invariance}\,, we quickly obtain the following fact. 

\begin{cor} \label{cor:q_standard_invariance} 
If $\xi\in\C^{\!2}$ and $s$ is a section of $E$ then the vector field $B(\xi)(s)$\,, which at any $u\in H\setminus0$ is equal to $B(\xi)(u,s_{\p_H(u)})$\,,  
satisfies $\bigl[A,B(\xi)(s)\bigr]=B(A\xi)(s)$\,, for any $A\in\mathfrak{gl}(2,\C\!)$\,. 
\end{cor} 

\indent 
We now have a convenient way to write $\Cal_H=(\dif\!\psi_H)^{-1}\bigl(T^{0,1}(\mathcal{L}\setminus0)\bigr)$\,, namely, 
$\Cal_H=\mathcal{B}_H\oplus({\rm ker}\dif\!\p_H)^{0,1}$, where $\mathcal{B}_H$ is generated by $B(e_1)(s)$\,, for any (local) section $s$ of $E$, 
with $(e_1,e_2)$ the canonical basis of $\C^{\!2}$. 

\begin{prop} \label{prop:q_h_harmonic} 
Let $(M,F,\r,c)$ be a quaternionic object with twistor space $Z$, and a decomposition $F^{\C\!}=H\otimes E$ 
corresponding to a holomorphic line bundle $\mathcal{L}$ over $Z$. Endow $H\setminus0$ with a principal $\r$-connection $C$ as in 
Proposition \ref{prop:q_a_connection}\,.\\ 
\indent 
Then a section $f$ of $\Lambda^2H$ is quaternionic pluriharmonic if and only if the corresponding equivariant function $\widetilde{f}$ on $H\setminus0$ 
satisfies $\bigl(\bigl(\overline{\partial}\circ J\circ\partial\bigr)(\widetilde{f}\,)\bigr)(X,Y)=0$\,, for any $X,Y\in\mathcal{B}_H$\,. 
\end{prop} 
\begin{proof} 
The necessity of that condition for $f$ to be quaternionic pluriharmonic is trivial. For the sufficiency, firstly, note that 
$f$ is quaternionic pluriharmonic if and only if 
\begin{equation} \label{e:q_h_harmonic} 
X\bigl((JY)(\widetilde{f}\,)\bigr)-Y\bigl((JX)(\widetilde{f}\,)\bigr)-\bigl(J[X,Y]\bigr)(\widetilde{f}\,)=0\;, 
\end{equation} 
for any sections $X$ and $Y$ of $\Cal_H$\,.\\ 
\indent 
By writing $\widetilde{f}$ with respect to a local section of $H\setminus0$\,, a straightforward calculation shows that \eqref{e:q_h_harmonic} always holds for any 
$X$ and $Y$ sections of $({\rm ker}\dif\!\p_H)^{0,1}$.\\ 
\indent 
Note that, $({\rm ker}\dif\!\p_H)^{0,1}$ is generated by (the fundamental complex vector fields corresponding to) those $A\in\mathfrak{gl}(2,\C\!)$ 
for which the first column is zero. Take $X$ to be such an $A$\,, and $Y=B(e_1)(s)$ for some section $s$ of $E$. Then Corollary \ref{cor:q_standard_invariance} 
shows that $[X,Y]=0$\,. Further, $JX$ is given by the matrix whose first column is the opposite of the second column of $A$\,, whilst its second column is zero. 
Also, $JY=B(e_2)(s)$\,.\\ 
\indent 
On denoting by $\xi\bigl(\in\C^{\!2}\bigr)$ the second column of $A$ we have that $(JX)(\widetilde{f}\,)=\xi_1\widetilde{f}$. Hence, 
$Y\bigl((JX)(\widetilde{f}\,)\bigr)=\bigl(B(\xi_1e_1)(s)\bigr)(\widetilde{f}\,)$\,.\\ 
\indent 
Now, by applying, again, Corollary \ref{cor:q_standard_invariance}\,, we obtain that $[X,JY]=B(\xi)(s)$\,. Hence,  
\begin{equation*} 
\begin{split} 
X\bigl((JY)(\widetilde{f}\,)\bigr)&=JY\bigl((X)(\widetilde{f}\,)\bigr)+\bigl(B(\xi)(s)\bigr)(\widetilde{f}\,)\\ 
&=-\xi_2\bigl(B(e_2)(s)\bigr)(\widetilde{f}\,)+\bigl(B(\xi)(s)\bigr)(\widetilde{f}\,)=\bigl(B(\xi_1e_1)(s)\bigr)(\widetilde{f}\,)\;. 
\end{split} 
\end{equation*} 
\indent 
We have, thus, shown that \eqref{e:q_h_harmonic} is automatically satisfied, also, when $X$ is a section $({\rm ker}\dif\!\p_H)^{0,1}$  
and $Y$ is a section of $\mathcal{B}_H$\,. The proof is complete. 
\end{proof} 

\begin{rem} \label{rem:q_hyper_ph} 
A quick consequence of Proposition \ref{prop:q_h_harmonic} is that, for a hypercomplex object, the quaternionic pluriharmonic sheaf 
is the intersection of the hypercomplex pluriharmonic sheaves, with respect to all possible (anti-commuting) pairs $(I,J)$\,.\\ 
\indent  
Consequently, the following assertions are equivalent, for a hypercomplex object:\\ 
\indent 
\quad(i) The hypercomplex pluriharmonicity does not depend of the pair $(I,J)$\,.\\ 
\indent 
\quad(ii) The hypercomplex and the quaternionic pluriharmonicities coincide. 
\end{rem} 

\indent 
Next, we give the following: 

\begin{proof}[Proof of Theorem \ref{thm:q_Penrose}] 
To prove (i)\,, by passing, through pull back, to $H\setminus0$\,, we may suppose $M$ hypercomplex and, on it, the 
sheaf of quaternionic pluriharmonic sections of $\Lambda^2H$ is equal to the sheaf of hypercomplex pluriharmonic sections of $\Lambda^2H$. 
Then, by using \v Cech cohomology, integral formulae similar to \eqref{e:Penrose_prehistory} can be, locally, obtained 
for the sections of $\Lambda^2H$ which are in the image of the Penrose transform. Consequently, the image 
of the Penrose transform is contained in the space of quaternionic pluriharmonic sections of $\Lambda^2H$.\\ 
\indent 
To prove (ii)\,, we may complexify $M$ so that $H$ and (consequently) $Y$ are the restrictions to $M$ of holomorphic bundles over $M^{\C\!}$. 
For simplicity, we denote with the same $H$ and $Y$, respectively, these bundles over $M^{\C\!}$; also, the same for the notation of the corresponding 
bundle projections. We, thus, have the same diagram as \eqref{diagram:for_Penrose}\,, up to the fact that $M$ is replaced by $M^{\C\!}$. 
Furthermore, we have the following diagram (which, in part, is the complexification of \eqref{diagram:for_Penrose}\,),  
\begin{equation}  \label{diagram:for_Penrose_complex} 
\begin{gathered} 
\xymatrix{
                                                       &              {\rm GL}(H) \ar[dl]_{\psi_{H,j}} \ar[d]^p  \ar@/^/[ddr]^{\p_H}       &                  \\
    \mathcal{L}\setminus0  \ar[d]   &                     \hspace{3mm}(Y+Y)\setminus Y    \ar[dl]_{\psi_j} \ar[dr]^{\p}      &                  \\
                      Z                               &                                                                 &            M^{\C}                             
   } 
\end{gathered} 
\end{equation} 
where ${\rm GL}(H)$ is the frame bundle of $H$, $Y(=PH)$ is diagonnally embedded into $Y+Y$, whilst $\psi_{H,j}$ is the composition 
of $\psi_H$ with the projection onto the $j$-component of any frame on $H$, and similarly for $\psi_j$\,, $(j=1,2)$\,, and we have, also, denoted 
by $\p_H$ the projection from ${\rm GL}(H)$ onto $M^{\C\!}$.\\ 
\indent 
Let $\partial_j$ be the partial exterior differential determined by the holomorphic foliation $\Cal_j$ given by (the connected components of) the fibres 
of $\psi_{H,j}$\,, $(j=1,2)$\,. As ${\rm GL}(H)$ is the complexification of $(H\setminus0)|_M$\,, we obtain (through a complexification) 
a morphism of holomorphic vector bundles $J:\Cal_2^*\to\Cal_1^*$. 
Then a holomorphic section $f$ of $\Lambda^2H$ is (complex-)quaternionic pluriharmonic if and only if the corresponding equivariant holomorphic function 
$\widetilde{f}$ on ${\rm GL}(H)$ satisfies $(\partial_1\circ J\circ\partial_2)(\widetilde{f}\,)=0$\,.\\ 
\indent 
Note that, ${\rm GL}(H)$\,, as a bundle over $(Y+Y)\setminus Y$, is a principal bundle with group $\C^{\!*}\times\C^{\!*}$, where the action 
is induced, by restriction to the diagonal matrices, from the action of ${\rm GL}(2,\C\!)$\,. Furthermore, 
$\psi_{H,j}$ is a morphism of principal bundles, covering $\psi_j$\,, with respect to the morphism of Lie groups $\C^{\!*}\times\C^{\!*}\to\C^{\!*}$, 
$(\l_1,\l_2)\mapsto\l_j$\,, $(j=1,2)$\,.\\ 
\indent 
Let $f$ be a section of $\Lambda^2H$ and let $\widetilde{f}$ be the corresponding equivariant function on ${\rm GL}(H)$\,. 
Then, for any $(\l_1,\l_2)\in\C^{\!*}\times\C^{\!*}$, we have 
\begin{equation} \label{e:cocycle_invariance} 
R_{(\l_1,\l_2)}^*\bigl((J\circ\partial_2)(\widetilde{f}\,)\bigr)=\l_1^{-2}(J\circ\partial_2)(\widetilde{f}\,)\;, 
\end{equation}
where $R^*$ denotes the pull back transformation induced by the right translation. This is straightforward if we restrict to the fibres of $\p_H$\,. 
Further, the proof of Proposition \ref{prop:q_a_connection} works in this setting as well by passing, if necessary, to an open neighbourhood 
of each point of $M$. Consequently, Proposition \ref{prop:q_standard_invariance} can be easily adapted to this setting, and by using it 
\eqref{e:cocycle_invariance} quickly follows.\\ 
\indent 
Consequently, if $f$ is a quaternionic pluriharmonic section of $\Lambda^2H$ then the partial $1$-form $(J\circ\partial_2)(\widetilde{f}\,)$ 
defines a $1$-cocycle,  for the relative de Rham cohomology (see \cite{EGW}\,) of $\mathcal{L}^2$, with respect to $\psi_1$\,, whose first group we denote by 
$H_{\psi_1\!}^1\bigl(Z,\mathcal{L}^2\bigr)$\,. Consequently, we have obtained an injective linear map, into 
this cohomology group, from the space of quaternionic pluriharmonic sections of $\Lambda^2H$ (on integrating $(J\circ\partial_2)(\widetilde{f}\,)$ 
along the fibres of $\p_H$ we obtain $f$, up to a nonzero constant factor).\\ 
\indent 
Therefore (ii) is proved if, locally, 
we have $H^1\bigl(Z,\mathcal{L}^2\bigr)=H_{\psi_1\!}^1\bigl(Z,\mathcal{L}^2\bigr)$\,. From the main result of \cite{EGW}\,, we deduce that, locally, 
this holds if the fibres of $\psi$ are contractible. Now, if $(M,F,\r,c)$ is of constant type, we may assume that the 
images through $\p_H$ of the fibres of $\psi$ are totally geodesic with respect to a (classical) connection on $M^{\C\!}$. Thus, in this case, 
by passing to any convex neighbourhood over which $Y$ is trivial, we have the desired cohomology isomorphism, and the proof is complete. 
\end{proof} 

\begin{rem} 
From Proposition \ref{prop:qPt_inj} and Theorem \ref{thm:q_Penrose} we obtain that, locally, the Penrose transform is an isomorphism in the following cases:\\ 
\indent(1) for (classical) quaternionic manifolds \cite{Bas-q_complexes} (see, also, \cite{Hit-80}\,, for the particular case of anti-self-dual manifolds).\\ 
\indent(2) for co-CR quaternionic vector spaces (this was, essentially, proved in \cite{Tsai}\,). 
\end{rem}

\end{document}